\documentclass[pdflatex,sn-mathphys-num]{sn-jnl}
\usepackage{graphicx}%
\usepackage{multirow}%
\usepackage{amsmath,amssymb,amsfonts}%
\usepackage{amsthm}%
\usepackage{mathrsfs}
\usepackage{mathrsfs}
\usepackage[title]{appendix}%
\usepackage{xcolor}%
\usepackage{textcomp}%
\usepackage{manyfoot}%
\usepackage{booktabs}%
\usepackage{algorithm}%
\usepackage{algorithmicx}%
\usepackage{algpseudocode}%
\usepackage{listings}%
\usepackage{framed}
\usepackage{nameref}
\usepackage{orcidlink}
\usepackage{amsmath, amsthm, amssymb, amsfonts, mathrsfs}
\usepackage{enumerate}
\usepackage{geometry}
\usepackage{mathtools}
\usepackage{hyperref}
\usepackage{cleveref}
\theoremstyle{plain}
\newtheorem{theorem}{Theorem}[section]
\newtheorem{proposition}[theorem]{Proposition}
\newtheorem{lemma}[theorem]{Lemma}
\newtheorem{corollary}[theorem]{Corollary}
\theoremstyle{definition}
\newtheorem{definition}[theorem]{Definition}
\newtheorem{rem}[theorem]{Remark}
\newtheorem{example}[theorem]{Example}
\newtheorem*{theorem*}{Theorem}

\DeclareMathOperator{\ext}{\textup{ext}}

\DeclareMathOperator{\inter}{\textup{int}}
\DeclareMathOperator{\Per}{\textup{Per}}

\DeclareMathOperator{\dist}{\textup{dist}}
\DeclareMathOperator{\diam}{\textup{diam}}
\DeclareMathOperator{\difm}{\textup{d$\mathcal{H}$}}

\begin{document}

\title[Existence of Anisotropic Minkowski Content]{Existence of Anisotropic Minkowski Content}

\author*[]{\fnm{Filip} \sur{Fryš}\,\orcidlink{0009-0000-2134-0109}}\email{filip.frys@matfyz.cuni.cz}\affil[]{\orgdiv{Mathematical Institute of Charles University}, \orgname{Charles University, Faculty of Mathematics and Physics}, \orgaddress{\street{Sokolovsk\'{a} 83}, \city{Praha 8}, \postcode{186 75}, \state{Czech Republic}}}


\abstract{We study the existence of anisotropic Minkowski content and anisotropic outer Minkowski content for sets of finite perimeter in an open set $\Omega\subseteq\mathbb{R}^n$. We prove that, for a set $E$ of finite perimeter in $\Omega$, the Minkowski content of $\partial E$ in $\Omega$ coincides with the perimeter of $E$ in $\Omega$ if and only if the anisotropic Minkowski content of $\partial E$ in $\Omega$ coincides with the arithmetic mean of the anisotropic perimeters of $E$ and $\Omega\setminus E$. More generally, we show that this existence property is independent of the chosen anisotropy. As a consequence, if $\overline{E}\subseteq\Omega$ and the anisotropic outer Minkowski contents of both $E$ and $\Omega\setminus E$ in $\Omega$ exist, then the corresponding isotropic outer Minkowski contents exist as well.}



\maketitle\section*{Acknowledgments}
The author is grateful to Jan Rataj for his guidance and valuable comments. This paper grew in part out of the author’s master’s thesis. The author also acknowledges the financial support of SVV-2025-260837, PRIMUS/24/SCI/009, and GAUK 34126.

\section*{Introduction}
\label{secIntro}

Let $\emptyset \neq S \subseteq \mathbb{R}^n$ be a measurable set. The \emph{Minkowski content of $S$} is defined by
\[
\mathcal{M}(S; \mathbb{R}^n)\coloneq
\lim_{\varepsilon\to0_+}
\frac{\lambda^n\big(\{x\in\mathbb{R}^n:\dist(x,S)<\varepsilon\}\big)}{2\varepsilon},
\]
whenever the limit exists, where $\lambda^n$ denotes the $n$-dimensional Lebesgue measure and $\dist(\cdot,S)$ the Euclidean distance from $S$. 

For two sets $A,B\subseteq\mathbb{R}^n$, we write
\[
A\oplus B \coloneq \{a+b : a\in A,\ b\in B\}
\]
for their Minkowski sum. Since
\[
\lambda^n\big(\{x\in\mathbb{R}^n:\dist(x,S)<\varepsilon\}\big)=\lambda^n\big(S\oplus B(0,\varepsilon)\big),
\]
the Minkowski content of $S$ can equivalently be written as
\[
\mathcal{M}(S; \mathbb{R}^n)
=
\lim_{\varepsilon\to0_+}
\frac{\lambda^n\big(S\oplus B(0,\varepsilon)\big)}{2\varepsilon},
\]
whenever the limit exists. If $S$ is the boundary of a convex body, then this limit exists and equals the $(n-1)$-dimensional Hausdorff measure of $S$; see \cite[Chapter 5]{rolf}.

A natural problem is to determine for which more general sets the Minkowski content exists and coincides with the expected surface measure. Classical results show that this is the case for compact $(n-1)$-rectifiable sets, and more generally for compact countably $\mathcal{H}^{n-1}$-rectifiable sets satisfying a suitable lower density condition; see \cite[p.~275]{federer} and \cite[Theorem 2.104]{ambrosio}.

It is also natural to localize the construction. If $\Omega\subseteq\mathbb{R}^n$ is open, one defines the \emph{Minkowski content of $S$ in $\Omega$} by
\begin{equation}\label{kek}
\mathcal{M}(S;\Omega)\coloneq
\lim_{\varepsilon\to0_+}
\frac{\lambda^n\Big(\big((S\cap\Omega)\oplus B(0,\varepsilon)\big)\cap\Omega\Big)}{2\varepsilon},
\end{equation}
whenever the limit exists.

A further natural generalization is obtained by replacing the Euclidean ball by an arbitrary convex body. Let $C\subseteq\mathbb{R}^n$ be a convex body with $0\in\inter(C)$. The corresponding \emph{$C$-anisotropic Minkowski content of $S$ in $\Omega$} is
\[
\mathcal{M}_C(S;\Omega)\coloneq
\lim_{\varepsilon\to0_+}
\frac{\lambda^n\Big(\big((S\cap\Omega)\oplus \varepsilon C\big)\cap\Omega\Big)}{2\varepsilon},
\]
whenever the limit exists. In the same spirit, the \emph{$C$-anisotropic outer Minkowski content} of a measurable set $E\subseteq\Omega$ is defined by
\[
\mathcal{SM}_C(E;\Omega)\coloneq
\lim_{\varepsilon\to0_+}
\frac{1}{\varepsilon}
\lambda^n\Big(\big((E\cap\Omega)\oplus\varepsilon C\big)\cap(\Omega\setminus E)\Big),
\]
whenever the limit exists. Its isotropic counterpart, corresponding to $C=B(0,1)$, was introduced in \cite{Villa4}, while the anisotropic version was studied in \cite{chambolle, Villa2}. 

Minkowski-type volume expansions and parallel enlargements also play an important role in stochastic geometry; see, for instance, \cite{KR06, HL00}. In particular, outer Minkowski content appears in the study of specific area, mean boundary densities, and related quantities for random closed sets and Boolean models; see \cite{Villa4, VillaMean, VillaSpecific}.

\medskip

For a set of finite perimeter $E$ in $\Omega$, the expected first-order terms are expressed by anisotropic perimeter functionals. More precisely, one compares $\mathcal{M}_C(\partial E;\Omega)$ with
\[
\frac{1}{2}\int_{\Omega\cap\partial^*E}\big(h_C(\nu_E)+h_C(-\nu_E)\big)\,d\mathcal{H}^{n-1},
\]
and $\mathcal{SM}_C(E;\Omega)$ with
\[
\int_{\Omega\cap\partial^*E}h_C(\nu_E)\,d\mathcal{H}^{n-1},
\]
where $\nu_E$ denotes the measure-theoretic outer unit normal, $\partial^* E$ stands for the \emph{reduced boundary of $E$}, and $h_C$ is the support function of $C$. For convex bodies these formul{\ae}  are classical; see again \cite[Chapter 5]{rolf}. In the anisotropic setting, existence results under rectifiability and density assumptions were obtained in \cite{Villa2}.

The main point of this paper is not merely to identify the correct first-order term, but to understand the role of the anisotropy itself. A basic question is whether the validity of the expected Minkowski formula is genuinely dependent on the chosen structuring body $C$, or whether it is instead an intrinsic property of the underlying set. Our main result shows that, for the two-sided Minkowski content of the topological boundary of a set of finite perimeter, the answer is the latter: existence for one full-dimensional anisotropy is equivalent to existence for every other. In this sense, the theorem can be viewed as a reduction principle: once the formula is known for one convenient anisotropy, it automatically follows for all others. This is useful in situations where one particular anisotropy is technically easier to handle, or where the isotropic case is already understood and one wants to transfer the conclusion to general anisotropies.

More precisely, in Theorem \ref{theoremExistence} we prove that for a set of finite perimeter $E$ in an open set $\Omega$ and convex body $C$,
\[
\mathcal{M}(\partial E;\Omega)=\mathcal{H}^{n-1}(\Omega\cap\partial^*E)
\]
holds if and only if
\[
\mathcal{M}_C(\partial E;\Omega)
=
\frac{1}{2}
\int_{\Omega\cap\partial^*E}\big(h_C(\nu_E)+h_C(-\nu_E)\big)\,d\mathcal{H}^{n-1}.
\]
The proof combines a decomposition of the reduced boundary into Lipschitz pieces, a local formula from \cite[Lemma 3.3]{Villa2}, and a global covering argument based on the Besicovitch covering theorem.

Our result is also related to the following theorem of Chambolle, Lisini and Lussardi \cite[Theorem 3.4]{chambolle}:
\begin{theorem*}[\cite{chambolle}]
Let $\Omega\subseteq\mathbb{R}^n$ be open, let $E\subseteq\Omega$ be a set of finite perimeter in $\Omega$, and let $C$ be a convex body whose interior contains the origin. If
\[
\mathcal{SM}_{B(0,1)}(E;\Omega)=\mathcal{H}^{n-1}(\Omega\cap\partial^*E),
\]
then
\[
\mathcal{SM}_C(E;\Omega)=\int_{\Omega\cap\partial^*E}h_C(\nu_E)\,d\mathcal{H}^{n-1}.
\]
\end{theorem*}

The converse was suggested in \cite[Remark 3.5]{chambolle}, but no proof was given. Using Theorem \ref{theoremExistence}, we obtain a partial converse of a slightly different form: under the additional assumption $\overline{E}\subseteq\Omega$, if both anisotropic outer Minkowski contents
\[
\mathcal{SM}_C(E;\Omega)
\quad\text{and}\quad
\mathcal{SM}_C(\Omega\setminus E;\Omega)
\]
have the expected perimeter values, then the corresponding isotropic outer Minkowski contents also have the expected values. Thus, under this natural localization assumption, the validity of the outer Minkowski formul{\ae} for $E$ and its complement is likewise independent of the chosen anisotropy.

The paper is organized as follows. In Section~\ref{secPrelim}, we fix notation and recall the necessary background from geometric measure theory and the theory of anisotropic Minkowski content. In Section~\ref{secFurtherResults}, we establish several basic properties of anisotropic Minkowski and outer Minkowski contents and compare the relevant functionals. Finally, in Section~\ref{secMainResult}, we prove Theorem~\ref{theoremExistence} and study the relation between the anisotropic Minkowski content of the reduced boundary and the anisotropic outer Minkowski content of the sets $E^1$ and $E^0$; see Proposition~\ref{reduced}.

\section{Preliminaries}
\label{secPrelim}
\subsection{Notation}
 Let $n\in\mathbb{N}$ (so $n\geq1$). For $x, y\in\mathbb{R}^n$, we denote by $x\cdot y=\sum_{i=1}^{n}x_i y_i$ the \textit{Euclidean scalar product of $x$ and $y$} and by $|\cdot|$ the induced \emph{Euclidean norm}.

If $A, B\subseteq\mathbb{R}^n$ are arbitrary sets and $E\subseteq\mathbb{R}^n$ is measurable, then: \begin{itemize} 
\item $A\oplus B\coloneq \{a+b:a\in A,\; b\in B\}$ denotes the \textit{Minkowski sum of $A$ and $B$};
    \item for $r\in\mathbb{R}$, the set $rA\coloneq\{ra:a\in A\}$ is the \textit{$r$-multiple of $A$};
    \item $A\Delta B\coloneq (A\setminus B)\cup(B\setminus A)$ denotes the \textit{symmetric difference of $A$ and $B$}; 
    \item $\inter (A)$, $\ext (A)$, $\partial A$, $\overline{A}$ denote the \textit{interior of $A$}, the \textit{exterior of $A$}, the \textit{boundary of $A$} and the \textit{closure of $A$}, respectively;
    \item if $A\neq\emptyset$, then $\dist(\cdot, A)$ denotes the classical \textit{Euclidean distance from $A$} and $\diam(A)$ denotes the \textit{diameter of $A$};
    \item $\lambda^n(E)$ denotes the \textit{Lebesgue measure of $E$} and $\mathcal{H}^s(A)$ the \textit{$s$-dimensional Hausdorff measure of $A$}, where $s\geq0$;
    \item $\chi_A$ denotes the \textit{characteristic function of $A$}.
\end{itemize}

Further, let $B(x,r)$ denote the (closed) ball centered at $x$ with radius $r>0$. 

Let $\mu$ be a nonnegative measure, and let $B$ be a $\mu$-measurable set.  
The symbol $\mu\Big|_B$ stands for the restriction of $\mu$ to $B$.

We abbreviate “almost everywhere” by a.e. and “almost all” by a.a.  

The space $\mathcal{C}_c^{\infty}(\Omega)$ consists of all infinitely differentiable functions on $\Omega$ with compact support.  

Let $\{E_\varepsilon\}_{\varepsilon>0}$ be a family of $\lambda^n$-measurable subsets of an open set $\Omega\subseteq\mathbb{R}^n$.  
We say that $\{E_\varepsilon\}_{\varepsilon>0}$ \emph{converges to} a set $E$ \emph{in} $L^1_{\mathrm{loc}}(\Omega)$ if  
\[
\chi_{E_\varepsilon} \xrightarrow[\varepsilon\to0_+]{} \chi_E 
\quad\text{in } L^1_{\mathrm{loc}}(\Omega).
\]

The symbol $\nabla$ will be used for the total derivative, and we identify the derivative with its representing matrix whenever convenient.

\subsection{Geometric Measure Theory}

In this subsection, we recall several useful notions and results from geometric measure theory. 

\begin{definition}\label{rectif}
Let $k \in \mathbb{N}_0$ with $k \leq n$, and let $S \subseteq \mathbb{R}^n$ be $\mathcal{H}^k$-measurable.  
We say that $S$ is \textit{countably $\mathcal{H}^{k}$-rectifiable} if there exists a sequence $\{f_i\}_{i=0}^{\infty}$ of Lipschitz mappings $f_i \colon \mathbb{R}^k \to \mathbb{R}^n$ such that 
\[
\mathcal{H}^k\!\left(S \setminus \bigcup_{i \in \mathbb{N}_0} f_i(\mathbb{R}^k)\right) = 0.
\]
\end{definition}

\begin{rem}\label{R1}
Let $S \subseteq \mathbb{R}^n$ be a countably $\mathcal{H}^{n-1}$-rectifiable set with $\mathcal{H}^{n-1}(S) < \infty$.  
Then there exists a countable family of pairwise disjoint compact subsets of $S$ that covers $S$ up to an $\mathcal{H}^{n-1}$-negligible set, each of which is contained in an $(n-1)$-dimensional Lipschitz graph and has finite $\mathcal{H}^{n-1}$-measure (see \cite[Lemma 11.1 and Remark 11.7]{simon}). 
\end{rem}

We next recall the notion of a set of finite perimeter.

\begin{definition}
Let $\Omega \subseteq \mathbb{R}^n$ be open and $E \subseteq \Omega$ measurable.  
We say that $E$ is a \textit{set of finite perimeter in $\Omega$} if the distributional derivative of $\chi_E$, denoted by $D\chi_E$, is an $\mathbb{R}^n$-valued Radon measure with finite total variation in $\Omega$, that is,
\[
\int_E \nabla f \, \mathrm{d}x = - \int_\Omega f \, \mathrm{d}(D\chi_E)
\quad \text{for all } f \in \mathcal{C}_c^\infty(\Omega).
\]
The \textit{perimeter of $E$ in $\Omega$} is defined by
\[
\Per(E; \Omega) \coloneq \|D\chi_E\|(\Omega),
\]
that is, the total variation of $D\chi_E$ in $\Omega$.
\end{definition}

\begin{definition}\label{outer}
Let $E \subseteq \mathbb{R}^n$ be measurable and $x \in \mathbb{R}^n$.  
We define the \textit{upper} and \textit{lower Lebesgue densities} of $E$ at $x$ by
\begin{align*}
\Theta_n(E, x)_* &\coloneq \liminf_{r \to 0_+} \frac{\lambda^n|_E(B(x, r))}{\lambda^n(B(x, r))}, 
&
\Theta_n(E, x)^* &\coloneq \limsup_{r \to 0_+} \frac{\lambda^n|_E(B(x, r))}{\lambda^n(B(x, r))}.
\end{align*}
If the two limits coincide, we denote their common value by $\Theta_n(E, x)$ and call it the \textit{Lebesgue density} of $E$ at $x$.

\smallskip
For $t \in [0,1]$, define
\[
E^t \coloneq \{x \in \mathbb{R}^n : \Theta_n(E, x) = t\}.
\]
The \textit{essential boundary} of $E$ is then
\[
\partial_* E \coloneq \mathbb{R}^n \setminus (E^0 \cup E^1).
\]
We also define the \textit{reduced boundary} of $E$ by
\[
\partial^* E \coloneq 
\Big\{x \in E^{1/2} : 
\exists\, \nu_E(x) \in \mathbb{S}^{n-1} \text{ s.t. }
\frac{E - x}{\varepsilon} \xrightarrow[\varepsilon \to 0_+]{} 
\{y \in \mathbb{R}^n : y \cdot \nu_E(x) \le 0\}
\Big\},
\] where the $L^1_{\mathrm{loc}}(\mathbb{R}^n)$-convergence of sets is considered.
If $x \in \partial^* E$, the vector $\nu_E(x)$ is called the \textit{outer unit normal to $E$ at $x$}.
\end{definition}

We recall the following fundamental result (see \cite[Theorem 3.59]{ambrosio}):

\begin{theorem}\label{theorem}
Let $E \subseteq \mathbb{R}^n$ be measurable.  
Then $\partial^* E$ is countably $\mathcal{H}^{n-1}$-rectifiable. If, in addition, $E$ has finite perimeter, then $$\|D\chi_E\| = \mathcal{H}^{n-1}\!\big|_{\partial^* E},$$ and for $\mathcal{H}^{n-1}\!\big|_{\partial^* E}$-a.e. $x \in \partial^* E$ one has
\[
1 = \lim_{\varepsilon \to 0_+} 
\frac{\mathcal{H}^{n-1}\big(\partial^* E \cap B(x, \varepsilon)\big)}{\omega_{n-1} \varepsilon^{n-1}}.
\]
\end{theorem}

\begin{rem}\label{remarko}
We recall the following facts (see \cite[Section 3.3]{ambrosio}):
\begin{itemize}
\item If $E \subseteq \mathbb{R}^n$ is measurable, then $x \in \mathrm{int}(E)$ implies $\Theta_n(E, x) = 1$, while $x \in \mathrm{ext}(E)$ implies $\Theta_n(E, x) = 0$.  
Hence $\partial^* E \subseteq \partial_* E \subseteq \partial E$.  
Since the definition of $\partial^* E$ does not depend on the particular representative of $\chi_E$, we obtain
\[
\overline{\partial^* E} \subseteq \bigcap_{\{E' \subseteq \mathbb{R}^n : \lambda^n(E \Delta E') = 0\}} \partial E'.
\]
According to \cite[Proposition 12.19 and Remark 15.3]{ASE}, if $E$ has finite perimeter, then there exists a Borel set $E'$ such that $\lambda^n(E \Delta E') = 0$ and $\partial E' = \overline{\partial^* E}$.  
In fact, it follows from the proof that one may take $E' \coloneq E^1$.

\item If $E$ is a set of finite perimeter in $\Omega$, then
\[
\Per(E; \Omega) = \mathcal{H}^{n-1}(\Omega \cap \partial_* E)
= \mathcal{H}^{n-1}(\Omega \cap \partial^* E).
\]

\item If $\lambda^n\big((E \Delta F) \cap \Omega\big) = 0$, then $\Per(E; \Omega) = \Per(F; \Omega)$.

\item If $E$ is a set of finite perimeter in $\Omega$, then so is $\Omega \setminus E$, and we have
\[
\nu_E = -\nu_{\Omega \setminus E} 
\quad \left(\mathcal{H}^{n-1}\Big|_{\Omega \cap \partial^* E}\right)\text{-a.e.},
\]
\[
\mathcal{H}^{n-1}(\Omega \cap \partial^* E)
= \mathcal{H}^{n-1}(\Omega \cap \partial^*(\Omega \setminus E)),
\quad
\Per(E; \Omega) = \Per(\Omega \setminus E; \Omega).
\]
\end{itemize}
\end{rem}

Finally, we recall the Besicovitch covering theorem (see \cite[Subsection 1.5.2, Theorem 2]{evans}):

\begin{theorem}\label{bes}
Let $A \subseteq \mathbb{R}^n$ and $\rho \colon A \to (0, \infty)$ be a bounded function.  
Then there exist $N \in \mathbb{N}$, depending only on $n$, and an at most countable set $S \subseteq A$ such that
\begin{align*}
A &\subseteq \bigcup_{x \in S} B(x, \rho(x)), 
&
\sum_{x \in S} \chi_{B(x, \rho(x))} &\le N.
\end{align*}
\end{theorem}

         \subsection{Anisotropic (Outer) Minkowski Content}
This subsection summarizes fundamental results on the anisotropic (outer) Minkowski content, following \cite{chambolle} and \cite{Villa2}.

\begin{definition}\label{convexity}
A \textit{convex body} is a non-empty, compact, and convex subset of $\mathbb{R}^n$.  
The collection of all convex bodies whose interiors contain the origin is denoted by $\mathcal{C}^n_0$.

Let $C \in \mathcal{C}^n_0$.  
The \textit{support function} of $C$ is defined by
\[
h_C(y) \coloneq \sup_{x \in C} x \cdot y, \quad y \in \mathbb{R}^n.
\]
The \textit{polar function} of $C$ is given by
\[
h_C^{\circ}(x) \coloneq \sup_{\substack{y \in \mathbb{R}^n \\ h_C(y) \leq 1}} x \cdot y, \quad x \in \mathbb{R}^n.
\]
For a nonempty set $E \subseteq \mathbb{R}^n$, we define the \textit{$C$-anisotropic distance function to $E$} by
\[
\textup{dist}_C(x, E) \coloneq \inf_{y \in E} h_C^{\circ}(x - y), \quad x \in \mathbb{R}^n.
\]
\end{definition}

\begin{rem}
Let $C \in \mathcal{C}^n_0$.  
\begin{itemize}
    \item The \textit{polar set} of $C$ is defined by $C^{\circ} \coloneq \{ x \in \mathbb{R}^n : h_C(x) \leq 1 \}$.  
    Then $h_C^{\circ} = h_{C^{\circ}}$ and $(C^{\circ})^{\circ} = C$.
    \item The function $x \mapsto h_C(x)$ is sublinear (hence convex) and Lipschitz continuous, and satisfies $h_{-C}(x) = h_C(-x)$ for all $x \in \mathbb{R}^n$.  
    Moreover, $h_{aC} = a h_C$ for any $a > 0$.
    \item If $a C \subseteq B(0,1) \subseteq b C$ for some $a,b > 0$, then for all $\nu \in \mathbb{R}^n$,
    \begin{equation}\label{hurich}
        a\,h_C(\nu) \leq |\nu| \leq b\,h_C(\nu).
    \end{equation}
\end{itemize}
\end{rem}

\begin{definition}
Let $Q \subseteq \mathbb{R}^n$ with $0 \in Q$, let $\Omega \subseteq \mathbb{R}^n$ be open, and $\varepsilon > 0$.  
The \textit{$(\varepsilon, Q)$-anisotropic outer Minkowski content (in $\Omega$)} is defined for $\lambda^n$-measurable sets $E \subseteq \mathbb{R}^n$ by
\[
\mathcal{SM}_{\varepsilon, Q}(E; \Omega)
    \coloneq \frac{1}{\varepsilon}
    \lambda^n\!\Big( \big( (E \cap \Omega) \oplus \varepsilon Q \big) \cap (\Omega \setminus E) \Big).
\]
We then define
\[
\mathcal{SM}_Q(E; \Omega)_* \coloneq
    \liminf_{\varepsilon \to 0_+} \mathcal{SM}_{\varepsilon, Q}(E; \Omega),
    \qquad
\mathcal{SM}_Q(E; \Omega)^* \coloneq
    \limsup_{\varepsilon \to 0_+} \mathcal{SM}_{\varepsilon, Q}(E; \Omega).
\]
The former is the \textit{lower} and the latter the \textit{upper $Q$-anisotropic outer Minkowski content of $E$ in $\Omega$}.  
If these two limits coincide, their common value
\[
\mathcal{SM}_Q(E; \Omega) \coloneq \mathcal{SM}_Q(E; \Omega)_*
\]
is called the \textit{$Q$-anisotropic outer Minkowski content of $E$ in $\Omega$}.  
Whenever we write $\mathcal{SM}_Q(E; \Omega)$, we implicitly assume that the limit exists.
\end{definition}

The case $Q = B(0,1)$ is of particular importance.  
We recall a sufficient condition for the existence of the limit $\mathcal{SM}_{B(0,1)}(E; \mathbb{R}^n)$ (\cite[Theorem~3.1]{vrv}):

\begin{theorem}\label{closed}
Let $E \subseteq \mathbb{R}^n$ be closed with $\partial E$ countably $\mathcal{H}^{n-1}$-rectifiable and bounded.  
Assume that there exist $\gamma > 0$ and a probability measure $\eta$ absolutely continuous with respect to $\mathcal{H}^{n-1}$ such that
\[
\forall r > 0, \; \forall x \in \partial E:\quad
\eta\big(B(x, r)\big) \geq \gamma r^{n-1}.
\]
Then
\[
\mathcal{SM}_{B(0,1)}(E; \mathbb{R}^n)
    = \textup{Per}(E; \mathbb{R}^n)
      + 2\,\mathcal{H}^{n-1}\big((\partial E) \cap E^{0}\big).
\]
\end{theorem}

This result extends to general convex bodies $C \in \mathcal{C}^n_0$ (see \cite[Theorem~4.3]{Villa2}).

In this work, we are mainly concerned with the case where $C$ is a full-dimensional convex body with $0 \in \textup{int}(C)$ and $E$ is a set of finite perimeter.  
There exists an example of a set of finite perimeter in $\mathbb{R}^3$ for which the isotropic outer Minkowski content (i.e., for $C = B(0,1)$) does not exist, while the anisotropic version exists for every two-dimensional disc (\cite[Ex.~3]{Rataj}).

It is important to note that modifying $E \subseteq \mathbb{R}^n$ on a set of Lebesgue measure zero may alter its (anisotropic) outer Minkowski content.  
For instance, if $E = \mathbb{Q}$ and $\Omega = \mathbb{R}$, the value changes drastically.  
In \cite{chambolle}, an alternative functional is introduced that coincides with $\mathcal{SM}_C(\cdot; \Omega)$ for sets satisfying $E = E^1$.  
Nevertheless, in this paper, we adhere to the standard definition of $\mathcal{SM}_C(\cdot; \Omega)$.

We now define the $C$-anisotropic Minkowski content:

\begin{definition}\label{Minkow}
Let $S \subseteq \mathbb{R}^n$ be measurable and $Q \subseteq \mathbb{R}^n$ with $0 \in Q$.  
We define the \textit{$(\varepsilon, Q)$-anisotropic Minkowski content of $S$ in $\Omega$} by
\[
\mathcal{M}_{\varepsilon, Q}(S; \Omega)
    \coloneq \frac{1}{2\varepsilon}
    \lambda^n\!\Big( \big( (S \cap \Omega) \oplus \varepsilon Q \big) \cap \Omega \Big),
    \quad \varepsilon > 0.
\]
Then
\[
\mathcal{M}_Q(S; \Omega)_* \coloneq \liminf_{\varepsilon \to 0_+} \mathcal{M}_{\varepsilon, Q}(S; \Omega),
\qquad
\mathcal{M}_Q(S; \Omega)^* \coloneq \limsup_{\varepsilon \to 0_+} \mathcal{M}_{\varepsilon, Q}(S; \Omega).
\]
If these two limits coincide, the common value
\[
\mathcal{M}_Q(S; \Omega) \coloneq \mathcal{M}_Q(S; \Omega)_*
\]
is called the \textit{$Q$-anisotropic Minkowski content of $S$ in $\Omega$}.  
Whenever we write $\mathcal{M}_Q(S; \Omega)$, we implicitly assume that the limit exists.
\end{definition}
Let us point out that if $Q=B(0,1)$, the adjective \textit{anisotropic} is omitted (or replaced by \textit{isotropic}), and we write $\mathcal{SM}_{\varepsilon}$ instead of $\mathcal{SM}_{\varepsilon, B(0,1)}$, $\mathcal{SM}$ instead of $\mathcal{SM}_{B(0,1)}$, $\mathcal{M}_{\varepsilon}$ instead of $\mathcal{M}_{\varepsilon, B(0,1)}$, and $\mathcal{M}$ instead of $\mathcal{M}_{B(0,1)}$.    

\begin{definition}
Let $C\in\mathcal{C}^n_0$ and $\Omega\subseteq\mathbb{R}^n$ be open. For a measurable set $E\subseteq\mathbb{R}^n$, we define the \textit{$C$-anisotropic perimeter of $E$ in $\Omega$} by
\[
\textup{Per}_{h_C}(E;\Omega)\coloneq
\begin{cases}
\displaystyle \int_{\Omega\cap \partial^* E} h_C(\nu_E)\, d\mathcal{H}^{n-1}, & \text{if $\Per(E;\Omega)<\infty$},\\
+\infty, & \text{otherwise}.
\end{cases}
\]
\end{definition}

\begin{definition}
Let $\Omega\subseteq\mathbb{R}^n$ be open, $E\subseteq\Omega$ be measurable, and $C\in\mathcal{C}^n_0$. We say that the \textit{$C$-anisotropic outer Minkowski content of $E$ in $\Omega$ exists} if
\[
\mathcal{SM}_C(E;\Omega) = \textup{Per}_{h_C}(E;\Omega).
\]
Similarly, the \textit{$C$-anisotropic Minkowski content of $\partial E$ (or of $\partial^* E$) in $\Omega$ exists} if it coincides with
\[
\frac{1}{2}\big(\textup{Per}_{h_C}(E;\Omega) + \textup{Per}_{h_C}(\Omega\setminus E;\Omega)\big).
\]
\end{definition}

\begin{rem}\label{REMAk}
Suppose $E$ is a set of finite perimeter in an open set $\Omega\subseteq\mathbb{R}^n$, $C, C'\in\mathcal{C}^n_0$, and $S\subseteq\mathbb{R}^n$ is measurable. Then:
\begin{itemize}
    \item The $C$-anisotropic outer Minkowski content of $E$ in $\Omega$ is sometimes called one-sided. Intuitively, we enlarge $E$ only \emph{outwards}, whereas $\partial E$ is enlarged in both directions, \emph{inwards} and \emph{outwards}.
    \item The classical perimeter satisfies $\Per(E;\Omega) = \textup{Per}_{h_{B(0,1)}}(E;\Omega)$.
    \item If $a C\subseteq C'\subseteq b C$ for some positive constants $a$ and $b$, then
    \[
    a\,\textup{Per}_{h_C}(E;\Omega) \le \textup{Per}_{h_{C'}}(E;\Omega) \le b\,\textup{Per}_{h_C}(E;\Omega) \quad \text{(see (\ref{hurich}))},
    \]
    \[
    a\,\mathcal{M}_{a\varepsilon,C}(S;\Omega) \le \mathcal{M}_{\varepsilon,C'}(S;\Omega) \le b\,\mathcal{M}_{b\varepsilon,C}(S;\Omega), \quad \varepsilon>0,
    \]
    \[
    a\,\mathcal{SM}_{a\varepsilon,C}(E;\Omega) \le \mathcal{SM}_{\varepsilon,C'}(E;\Omega) \le b\,\mathcal{SM}_{b\varepsilon,C}(E;\Omega), \quad \varepsilon>0.
    \]
    Hence, these three functionals behave similarly. In particular, $\mathcal{SM}_C(E;\Omega)=0$ implies $\mathcal{SM}_{C'}(E;\Omega)=0$, and similarly $\mathcal{M}_C(S;\Omega)=0$ implies $\mathcal{M}_{C'}(S;\Omega)=0$.
    \item For $\varepsilon>0$ (see the discussion on page 5 in \cite{chambolle}):
    \begin{equation}\label{min4}
        \min_{\{F\subseteq \Omega : \lambda^n(E\Delta F)=0\}} \mathcal{SM}_{\varepsilon,C}(F;\Omega) = \mathcal{SM}_{\varepsilon,C}(E^1;\Omega) = \mathcal{SM}_{\varepsilon,C}(\Omega\setminus E^0;\Omega).
    \end{equation}
    \item For $\varepsilon>0$, the functionals $2\mathcal{M}_{\varepsilon,C}(\cdot;\Omega)$ and $\mathcal{SM}_{\varepsilon,C}(\cdot;\Omega)$ coincide on $\lambda^n$-negligible sets.
    \item For all $\varepsilon>0$, $\mathcal{M}_{\varepsilon,C}(S;\Omega) \ge \lambda^n(S\cap\Omega)/(2\varepsilon)$, and therefore if $\mathcal{M}_C(S;\Omega)^*<\infty$, then $\lambda^n(S\cap\Omega)=0$.
        \item For all $\varepsilon>0$,
    \[
    \mathcal{SM}_{\varepsilon, C}(E; \Omega)
    =\frac{1}{\varepsilon}\lambda^n\Big(\big((E\oplus \varepsilon C)\cap\Omega\big)\setminus E\Big)
    \geq
    \frac{1}{\varepsilon}\lambda^n\Big(\big(\overline{E}\cap\Omega\big)\setminus E\Big),
    \]
    because $\overline{E}\subseteq E\oplus\varepsilon C$. Consequently, if $\mathcal{SM}_{C}(E; \Omega)^*<\infty$, then
    \[
    \lambda^n\big((\overline{E}\cap\Omega)\setminus E\big)=0.
    \]
\end{itemize}
\end{rem}

\section{Basic Properties of Anisotropic (Outer) Minkowski Content}
\label{secFurtherResults}

In this section, we summarize basic properties of $\mathcal{M}_C(\cdot; \Omega)$ and $\mathcal{SM}_C(\cdot; \Omega)$, along with necessary conditions for their existence.

We start with the following lemma, which is crucial for the proof of our main result (cf. \cite[Lemma 3.3]{Villa2}).

\begin{lemma}\label{crucial}
    If $S$ is a compact set contained in a Lipschitz $(n-1)$-graph, then
    \[
        \mathcal{M}_{C}(S; \Omega)=\frac{1}{2}\int\limits_{S} \big(h_C(\nu_S)+h_C(-\nu_S)\big)\, \difm^{n-1}
    \]
    for any open set $\Omega$ with $S\subseteq\Omega\subseteq\mathbb{R}^n$, where $\nu_S\colon S\to\mathbb{S}^{n-1}$ is a unit normal to $S$ (see \cite[p. 3]{Villa2}). 
\end{lemma}
The next proposition provides lower bounds for the lower $C$-anisotropic Minkowski content and the lower $C$-anisotropic outer Minkowski content.
\begin{proposition}[Lower bounds]\label{eq:4}
Let $E$ be a set of finite perimeter in an open set $\Omega\subseteq\mathbb{R}^n$, and let $C\in\mathcal{C}^n_0$. Then
\begin{equation}
    \mathcal{SM}_{C}(E; \Omega)_*=\liminf_{\varepsilon\to 0_+} \mathcal{SM}_{\varepsilon,C}(E; \Omega) \geq \textup{Per}_{h_C}(E; \Omega),
\end{equation}
and
\begin{equation}
\begin{split}
    \mathcal{M}_{C}(\partial E; \Omega)_* &= \liminf_{\varepsilon\to 0_+} \mathcal{M}_{\varepsilon,C}(\partial E; \Omega) \geq \mathcal{M}_{C}(\partial^* E; \Omega)_* = \liminf_{\varepsilon\to 0_+} \mathcal{M}_{\varepsilon,C}(\partial^* E; \Omega)\\
    &\geq \frac{1}{2}\big(\textup{Per}_{h_C}(E; \Omega)+\textup{Per}_{h_C}(\Omega\setminus E; \Omega)\big).
\end{split}
\end{equation} 
\end{proposition}

\begin{proof}
The first part follows, for instance, from \cite[Theorem 3.1]{chambolle}.

For the second part, since $\partial^* E \subseteq \partial E$, we immediately have
\begin{equation}\label{eq:5}
    \mathcal{M}_{C}(\partial E; \Omega)_* \geq \mathcal{M}_{C}(\partial^* E; \Omega)_*.
\end{equation}
The set $\partial^* E \cap \Omega$ is countably $\mathcal{H}^{n-1}$-rectifiable with finite $\mathcal{H}^{n-1}$-measure. Therefore, there exists a sequence $\{S_k\}_{k=1}^{\infty}$ of pairwise disjoint compact subsets of $\partial^* E \cap \Omega$ that covers $\partial^* E \cap \Omega$ up to an $\mathcal{H}^{n-1}$-negligible set, with each $S_k$ contained in a graph of a Lipschitz function (see Remark \ref{R1}).

Fix an arbitrary $L\in\mathbb{N}$. Since the compact sets $S_1,\dots,S_L$ are pairwise disjoint, their mutual distances are positive. Therefore, for all sufficiently small $r>0$, the sets $S_k\oplus rC$, $k=1, \dots, L$, are pairwise disjoint. Applying Lemma \ref{crucial} yields
\begin{equation}\label{eq:6}
\begin{split}
    \sum_{k=1}^{L} \frac{1}{2}\int\limits_{S_k} \big(h_C(\nu_E)+h_C(-\nu_E)\big) \difm^{n-1} 
    &= \sum_{k=1}^L \mathcal{M}_{C}(S_k; \Omega)\\
    &= \sum_{k=1}^L \lim_{r\to0_+} \frac{1}{2r} \lambda^n \big((S_k \oplus r C) \cap \Omega\big)\\
    &= \lim_{r\to0_+} \frac{1}{2r} \lambda^n \Big(\Big(\bigcup_{k=1}^L S_k \oplus r C \Big) \cap \Omega \Big)\\
    &\leq \liminf_{r\to0_+} \frac{1}{2r} \lambda^n \Big(\Big((\partial^* E \cap \Omega) \oplus r C \Big) \cap \Omega \Big)\\
    &= \mathcal{M}_{C}(\partial^* E; \Omega)_*.
\end{split}
\end{equation}
Since $L$ was arbitrary, we deduce
\[
    \frac{1}{2} \big(\textup{Per}_{h_C}(E; \Omega)+\textup{Per}_{h_C}(\Omega\setminus E; \Omega)\big) \leq \mathcal{M}_{C}(\partial^* E; \Omega)_*,
\]
which together with (\ref{eq:5}) gives the desired inequality.
\end{proof}

\begin{rem}
The inequalities in Proposition~\ref{eq:4} may be strict even if the corresponding limits exist. For instance, in \cite[Example~1]{Villa4}, the authors present a compact set $E$ such that
\[
\Per(E; \mathbb{R}^n)<\mathcal{H}^{n-1}(\partial E)<\mathcal{SM}(E; \mathbb{R}^n).
\]
\end{rem}

\begin{rem}\label{EWQ}
A simple consequence of Proposition \ref{eq:4} is the following.

If
\[
\mathcal{SM}_C(E; \Omega)=\textup{Per}_{h_C}(E; \Omega),
\]
then, since $\lambda^n(E^1\Delta E)=0$, we also have
\[
\textup{Per}_{h_C}(E^1; \Omega)=\textup{Per}_{h_C}(E; \Omega),
\]
and therefore
\[
\textup{Per}_{h_C}(E^1; \Omega)\leq \mathcal{SM}_{C}(E^1; \Omega)_*
\leq \mathcal{SM}_{C}(E^1; \Omega)^*
\leq \mathcal{SM}_C(E; \Omega)
=\textup{Per}_{h_C}(E^1; \Omega).
\]
Hence
\[
\mathcal{SM}_{C}(E^1; \Omega)=\textup{Per}_{h_C}(E^1; \Omega)=\textup{Per}_{h_C}(E; \Omega).
\]

Likewise, if
\[
\mathcal{M}_C(\partial E; \Omega)=\frac{1}{2}\big(\textup{Per}_{h_C}(E; \Omega)+\textup{Per}_{h_C}(\Omega\setminus E; \Omega)\big),
\]
then
\begin{equation*}
\begin{split}
\frac{1}{2}\big(\textup{Per}_{h_C}(E; \Omega)+\textup{Per}_{h_C}(\Omega\setminus E; \Omega)\big)
&\leq \mathcal{M}_C(\partial^* E; \Omega)_*\\
&\leq \mathcal{M}_C(\partial^* E; \Omega)^*\\
&\leq \mathcal{M}_C(\partial E; \Omega)\\
&= \frac{1}{2}\big(\textup{Per}_{h_C}(E; \Omega)+\textup{Per}_{h_C}(\Omega\setminus E; \Omega)\big).
\end{split}
\end{equation*}
Hence
\[
\mathcal{M}_C(\partial^* E; \Omega)
=\frac{1}{2}\big(\textup{Per}_{h_C}(E; \Omega)+\textup{Per}_{h_C}(\Omega\setminus E; \Omega)\big).
\]
\end{rem}

\medskip

We recall the following lemma (cf. \cite[Lemma 1]{Villa4}):

\begin{lemma}\label{van}
Let $\{a_{\varepsilon}\}_{\varepsilon>0}$ and $\{b_{\varepsilon}\}_{\varepsilon>0}$ be families of non-negative real numbers. If there exist $a, b\in\mathbb{R}$ such that
\[
    \limsup_{\varepsilon\to0_+}(a_{\varepsilon}+b_{\varepsilon})\leq a+b, \quad 
    \liminf_{\varepsilon\to0_+}a_{\varepsilon}\geq a, \quad 
    \text{and} \quad \liminf_{\varepsilon\to0_+}b_{\varepsilon}\geq b,
\]
then
\[
    \lim_{\varepsilon\to0_+} a_{\varepsilon}=a \quad \text{and} \quad \lim_{\varepsilon\to0_+} b_{\varepsilon}=b.
\]
\end{lemma}

\medskip

\begin{definition}[\cite{chambolle}]
Let $\Omega\subseteq\mathbb{R}^n$ be open. For $\varepsilon>0$ and measurable $E\subseteq\Omega$, define
\[
    \mathscr{M}_{\varepsilon, C}(E; \Omega) \coloneqq \frac{1}{2} \Big(\mathcal{SM}_{\varepsilon, C}(E; \Omega) + \mathcal{SM}_{\varepsilon, C}(\Omega\setminus E; \Omega)\Big).
\]
Moreover, set
\[
    \mathscr{M}_{C}(E; \Omega)_* \coloneqq \liminf_{\varepsilon\to 0_+} \mathscr{M}_{\varepsilon, C}(E; \Omega), \quad
    \mathscr{M}_{C}(E; \Omega)^* \coloneqq \limsup_{\varepsilon\to 0_+} \mathscr{M}_{\varepsilon, C}(E; \Omega).
\]
If $\mathscr{M}_{C}(E; \Omega)^* = \mathscr{M}_{C}(E; \Omega)_*$, define
\[
    \mathscr{M}_{C}(E; \Omega) \coloneqq \mathscr{M}_{C}(E; \Omega)^*.
\]
We say that $\mathscr{M}_{C}(E; \Omega)$ \emph{exists} if
\[
    \mathscr{M}_{C}(E; \Omega) = \frac{1}{2}\Big(\textup{Per}_{h_C}(E; \Omega) + \textup{Per}_{h_C}(\Omega \setminus E; \Omega)\Big).
\]
For $C=B(0,1)$, we omit the subscript $C$.
\end{definition}

The functional $E\mapsto\mathscr{M}_{\varepsilon, C}(E; \Omega)$ averages the two one-sided outer expansions $E\mapsto\mathcal{SM}_{\varepsilon, C}(E; \Omega)$ and $E\mapsto\mathcal{SM}_{\varepsilon, C}(\Omega\setminus E; \Omega).$

\begin{definition}
For an open set $\Omega\subseteq \mathbb{R}^n$ and $\varepsilon>0$, define the functional
\[
    \mathfrak{M}_{\varepsilon, C}(S; \Omega) \coloneqq \frac{1}{2\varepsilon} \lambda^n \big((S \oplus \varepsilon C) \cap \Omega\big),
\]
where $S\subseteq\mathbb{R}^n$ is measurable. Its lower and upper limits are
\[
    \mathfrak{M}_{C}(S; \Omega)_* \coloneqq \liminf_{\varepsilon\to0_+} \mathfrak{M}_{\varepsilon, C}(S; \Omega), \quad
    \mathfrak{M}_{C}(S; \Omega)^* \coloneqq \limsup_{\varepsilon\to0_+} \mathfrak{M}_{\varepsilon, C}(S; \Omega).
\]
If $\mathfrak{M}_{C}(S; \Omega)_* = \mathfrak{M}_{C}(S; \Omega)^*$, we set
\[
    \mathfrak{M}_{C}(S; \Omega) \coloneqq \mathfrak{M}_{C}(S; \Omega)^*.
\]
If $E$ is a set of finite perimeter in $\Omega$, we say that $\mathfrak{M}_{C}(\partial E; \Omega)$  (or $\mathfrak{M}_{C}(\partial^* E; \Omega)$) \emph{exists} if
it coincides with \[ \frac{1}{2}\Big(\textup{Per}_{h_C}(E; \Omega) + \textup{Per}_{h_C}(\Omega \setminus E; \Omega)\Big).
\]
For $C=B(0,1)$, we omit the subscript $C$.
\end{definition}

The key difference between $\mathfrak{M}_{\varepsilon, C}(S; \Omega)$ and $\mathcal{M}_{\varepsilon, C}(S; \Omega)$ is that, in the definition of the former, the set $S$ is \emph{not} intersected with $\Omega$.

\begin{rem} 
Let $C\in\mathcal{C}^n_0$. 
\begin{itemize}
    \item For any $\lambda^n$-measurable subset $E\subseteq\mathbb{R}^n$, we always have the equality (see \cite[eq. (10)]{Rataj})
    \begin{align*}
        \lambda^n(E\oplus C)=\lambda^n\big(E\oplus \inter(C)\big),
    \end{align*}
    from which it follows that
    \begin{align}\label{kider}
        \lambda^n\big((E\oplus C)\cap\Omega\big)=\lambda^n\Big(\big(E\oplus \inter(C)\big)\cap\Omega\Big) \quad \text{for any measurable $\Omega\subseteq\mathbb{R}^n$.}
    \end{align}
    
    \item Let $\Omega\subseteq\mathbb{R}^n$ be open, $E\subseteq\mathbb{R}^n$ be nonempty and measurable, and $\varepsilon>0$. For any $x\in\Omega$, we have $\varepsilon>\textup{dist}_C(x, E)=\inf_{y\in E}h_C^{\circ}(x-y)$ if and only if there exists $y\in E$ such that $\varepsilon>h_C^{\circ}(x-y)$, which is equivalent to $(x-y)\in\varepsilon \inter(C)$, and this holds if and only if $x\in \big(E\oplus\varepsilon \inter(C)\big)\cap\Omega$. Hence,
    \begin{equation*}
        \big(E\oplus\varepsilon \inter(C)\big)\cap\Omega=\{x\in\Omega:\textup{dist}_C(x, E)<\varepsilon\},
    \end{equation*}
    and thus
    \begin{equation*}
        \lambda^n\Big(\big(E\oplus\varepsilon \inter(C)\big)\cap\Omega\Big)=\lambda^n\big(\{x\in\Omega:\textup{dist}_C(x, E)<\varepsilon\}\big).
    \end{equation*}
    Using (\ref{kider}) and the fact that multiplication by $\varepsilon$ is a homeomorphism, we deduce
    \begin{equation}
        \lambda^n\big((E\oplus\varepsilon C)\cap\Omega\big)=\lambda^n\big(\{x\in\Omega:\textup{dist}_C(x, E)<\varepsilon\}\big),
    \end{equation}
    and consequently, if $\partial E\neq\emptyset$
    \begin{align}\label{eq:54}
        \mathfrak{M}_{\varepsilon, C}(\partial E; \Omega)=\frac{1}{2\varepsilon}\lambda^n\big(\{x\in\Omega:\textup{dist}_C(x, \partial E)<\varepsilon\}\big),
    \end{align}
    while, if $E\cap\Omega\neq\emptyset$,
    \begin{align*}
        \mathcal{M}_{\varepsilon, C}(E; \Omega)=\frac{1}{2\varepsilon}\lambda^n\big(\{x\in\Omega:\textup{dist}_C(x, E\cap\Omega)<\varepsilon\}\big).
    \end{align*}

    Similarly, for $E\subseteq\Omega$,
    \begin{equation*}
        \big(E\oplus\varepsilon \inter(C)\big)\cap(\Omega\setminus E)=\{x\in\Omega\setminus E:\textup{dist}_C(x, E)<\varepsilon\},
    \end{equation*}
    and
    \begin{equation}\label{jelp}
        \lambda^n\big((E\oplus\varepsilon C)\cap(\Omega\setminus E)\big)=\lambda^n\big(\{x\in\Omega:\textup{dist}_C(x, E)\in(0, \varepsilon)\}\big)+\lambda^n\big((\overline{E}\setminus E)\cap\Omega\big),
    \end{equation}
    hence
    \begin{equation}\label{eq:9}
        \mathcal{SM}_{\varepsilon, C}(E; \Omega)=\frac{1}{\varepsilon}\lambda^n\big(\{x\in\Omega:\dist_C(x, E)\in(0, \varepsilon)\}\big)+\frac{1}{\varepsilon}\lambda^n\big((\overline{E}\setminus E)\cap\Omega\big).
    \end{equation} 

    \item Let $\Omega\subseteq\mathbb{R}^n$ be open and $E\subseteq\mathbb{R}^n$ be nonempty and measurable. For $\delta\in(0, 1)$, we have
    \begin{equation*}
        \begin{split}
    \{x\in\Omega:\textup{dist}_C(x, E)<\varepsilon(1-\delta)\}&\subseteq\{x\in\Omega:\textup{dist}_C(x, E)\leq\varepsilon\}\\&\subseteq\{x\in\Omega:\textup{dist}_C(x, E)<\varepsilon(1+\delta)\},        
        \end{split}
    \end{equation*}
    from which it follows that
    \[
    \lim_{\varepsilon\to0_+}\frac{1}{2\varepsilon}\lambda^n\big(\{x\in\Omega:\textup{dist}_C(x, E)<\varepsilon\}\big)
    \]
    exists if and only if
    \[
    \lim_{\varepsilon\to0_+}\frac{1}{2\varepsilon}\lambda^n\big(\{x\in\Omega:\textup{dist}_C(x, E)\leq\varepsilon\}\big)
    \]
    exists, and if either limit exists, they are equal.
\end{itemize}
\end{rem}
We next compare $\mathfrak{M}_{\varepsilon, C}$, $\mathscr{M}_{\varepsilon, C}$, and $\mathcal{M}_{\varepsilon, C}$. The following proposition, inspired by the proof of Proposition 9 in \cite{Rataj}, gives simple inequalities between these functionals.

\begin{proposition}\label{simonka}
Let $\Omega\subseteq\mathbb{R}^n$ be open and $C\in\mathcal{C}^n_0$. We have
\[
\mathfrak{M}_{\varepsilon, C}(\partial E; \Omega)\geq\mathscr{M}_{\varepsilon, C}(E; \Omega)\geq\mathcal{M}_{\varepsilon, C}(\partial E; \Omega) \quad \text{for any $\varepsilon>0$ and measurable $E\subseteq\Omega$.}
\]
If $\overline{E}\subseteq\Omega$, the three quantities coincide. In particular, if $\Omega=\mathbb{R}^n$, we have for any $\varepsilon>0$
\[
\mathfrak{M}_{\varepsilon, C}(\partial E; \mathbb{R}^n)=\mathscr{M}_{\varepsilon, C}(E; \mathbb{R}^n)=\mathcal{M}_{\varepsilon, C}(\partial E; \mathbb{R}^n), \quad \text{$E\subseteq\mathbb{R}^n$ measurable.}
\]
Moreover,
\[
\mathscr{M}_{C}(E; \Omega) \text{ exists } \iff \text{both $\mathcal{SM}_C(E; \Omega)$ and $\mathcal{SM}_C(\Omega\setminus E; \Omega)$ exist.}
\]
In particular, if $\overline{E}\subseteq\Omega$, it holds that
\[
\mathcal{M}_{C}(\partial E; \Omega) \text{ exists } \iff \text{both $\mathcal{SM}_C(E; \Omega)$ and $\mathcal{SM}_C(\Omega\setminus E; \Omega)$ exist.}
\]
\end{proposition}

\begin{proof}
Let $\varepsilon>0$ and $E\subseteq\Omega$ be measurable. Observe first that
\begin{align}\label{ali}
\big((E\oplus\varepsilon C)\cap (\Omega\setminus E)\big)\cup\Big(\big((\Omega\setminus E)\oplus\varepsilon C\big)\cap E\Big) \subseteq \big((\partial E)\oplus \varepsilon C\big)\cap\Omega,
\end{align}
because if $x\in(E\oplus\varepsilon C)\cap (\Omega\setminus E)$, then $x\in\Omega\setminus E$ and $(x-\varepsilon C)\cap E\neq\emptyset$. The convexity of $C$ ensures that $(x-\varepsilon C)\cap \partial E\neq\emptyset$, whence $x\in\big((\partial E)\oplus \varepsilon C\big)\cap\Omega$. Similarly, if $x\in\big((\Omega\setminus E)\oplus\varepsilon C\big)\cap E$, then $x\in\big((\partial E)\oplus \varepsilon C\big)\cap\Omega$. These observations imply (\ref{ali}).

If $x\in\Big(\big((\partial E)\cap\Omega\big)\oplus \varepsilon \inter(C)\Big)\cap\Omega$, then there exists some $\varepsilon_0\in(0, \varepsilon)$ such that $x\in\big((\partial E)\cap\Omega\big)\oplus \varepsilon_0 C$. In particular, $x\in (\overline{E}\cap\Omega)\oplus \varepsilon_0 C$ and $x\in \big(\overline{\Omega\setminus E}\cap\Omega\big)\oplus \varepsilon_0 C$. If $x\in \Omega\setminus E$, we have
\[
x\in \big((\overline{E}\cap\Omega)\oplus \varepsilon_0 C\big)\cap(\Omega\setminus E)\subseteq \big(E\oplus (\varepsilon-\varepsilon_0)C\oplus\varepsilon_0 C\big)\cap(\Omega\setminus E)\subseteq (E\oplus\varepsilon C)\cap(\Omega\setminus E),
\]
since $\overline{E}=\bigcap_{r>0}E\oplus rC$. By analogy, if $x\in E$, we have
\[
x\in \Big(\big(\overline{\Omega\setminus E}\cap\Omega\big)\oplus \varepsilon_0 C\Big)\cap E\subseteq\big((\Omega\setminus E)\oplus\varepsilon C\big)\cap E.
\]
Hence,
\begin{equation}\label{gn}
\Big(\big((\partial E)\cap\Omega\big)\oplus \varepsilon \inter(C)\Big)\cap\Omega\subseteq\big((E\oplus\varepsilon C)\cap(\Omega\setminus E)\big)\cup\Big(\big((\Omega\setminus E)\oplus\varepsilon C\big)\cap E\Big).
\end{equation}

Using (\ref{ali}), (\ref{gn}) and (\ref{kider}), we obtain
\[
\mathfrak{M}_{\varepsilon, C}(\partial E; \Omega)\geq\mathscr{M}_{\varepsilon, C}(E; \Omega)\geq\mathcal{M}_{\varepsilon, C}(\partial E; \Omega).
\]
Clearly, if $\overline{E}\subseteq\Omega$, we get equalities. The remaining statements follow from Lemma \ref{van} and \cite[Theorem 3]{Villa4}.
\end{proof}

\begin{rem}\label{implik}
Proposition \ref{simonka} yields the following chain of implications:
\begin{gather*} 
\mathfrak{M}_{C}(\partial E; \Omega)\text{ exists} \implies \mathscr{M}_{C}(E; \Omega)\text{ exists} \implies \mathcal{M}_{C}(\partial E; \Omega)\text{ exists},\\
\Updownarrow \\
\mathcal{SM}_C(E; \Omega) \text{ and } \mathcal{SM}_C(\Omega\setminus E; \Omega) \text{ exist}.
\end{gather*}
Thus among the three functionals, $\mathcal{M}_C(\cdot; \Omega)$ is the weakest existence notion.
\end{rem}

We now give a simple example of a set $E\subseteq\Omega$ for which
\[
\mathfrak{M}(\partial E; \Omega) > \mathscr{M}(E; \Omega) > \mathcal{M}(\partial E; \Omega).
\]

\begin{example}
Define
\begin{align*}
\Omega &\coloneq \{x \in \mathbb{R} : |x| \in [0, 1) \cup (1, 2)\},\\
E &\coloneq \{x \in \mathbb{R} : |x| \in (1, 2)\}.
\end{align*}
Then
\[
\mathfrak{M}(\partial E; \Omega) = 3,\qquad
\mathscr{M}(E; \Omega) = 2,\qquad
\mathcal{M}(\partial E; \Omega) = 0.
\]

Indeed, $\Omega\cap\partial E=\emptyset$, because $\partial E=\{\pm1, \pm2\}$. Hence
\[
\mathcal{M}(\partial E; \Omega)=0.
\]
Moreover, for sufficiently small $\varepsilon>0$,
\[
\mathcal{SM}_{\varepsilon}(E; \Omega)
=\frac{1}{\varepsilon}\lambda^1\big(\{x\in\mathbb{R} : 1-\varepsilon<|x|<1\}\big)=2,
\]
and
\[
\mathcal{SM}_{\varepsilon}(\Omega\setminus E; \Omega)
=\frac{1}{\varepsilon}\lambda^1\big(\{x\in\mathbb{R} : 1<|x|<1+\varepsilon\}\big)=2,
\]
so both limits are equal to $2$. Therefore,
\[
\mathscr{M}(E; \Omega)=2.
\]
Finally, for sufficiently small $\varepsilon>0$,
\[
\big((\partial E)\oplus B(0, \varepsilon)\big)\cap\Omega
=
\{x\in\mathbb{R}:1-\varepsilon<|x|<1+\varepsilon\}
\;\cup\;
\{x\in\mathbb{R}:2-\varepsilon<|x|<2\},
\]
and the two terms contribute $4\varepsilon$ and $2\varepsilon$, respectively. Hence
\[
\mathfrak{M}(\partial E; \Omega)=3.
\]
\end{example}

\medskip

\section{Main Result}
\label{secMainResult}
This section is devoted to the existence of the anisotropic Minkowski content of the topological boundary of a set of finite perimeter. The proof of the following theorem relies on the Besicovitch covering theorem and Lemma \ref{crucial}.

\begin{theorem}
\label{theoremExistence}
Let $E$ be a set of finite perimeter in an open set $\Omega\subseteq\mathbb{R}^n$, and let $C, C^{\prime}\in\mathcal{C}^n_0$. Then
\[
\mathcal{M}_C(\partial E; \Omega)=\frac{1}{2}\big(\textup{Per}_{h_{C}}(E; \Omega)+\textup{Per}_{h_{C}}(\Omega\setminus E; \Omega)\big)
\]
if and only if
\[
\mathcal{M}_{C^{\prime}}(\partial E; \Omega)=\frac{1}{2}\big(\textup{Per}_{h_{C^{\prime}}}(E; \Omega)+\textup{Per}_{h_{C^{\prime}}}(\Omega\setminus E; \Omega)\big). 
\]
\end{theorem}
\begin{proof}
Since the roles of $C$ and $C^{\prime}$ are symmetric, it suffices to prove one implication.

Suppose that
\[
\mathcal{M}_C(\partial E; \Omega)=\frac{1}{2}\big(\textup{Per}_{h_{C}}(E; \Omega)+\textup{Per}_{h_{C}}(\Omega\setminus E; \Omega)\big).
\]
If $\mathcal{H}^{n-1}(\Omega\cap\partial^*E)=0$, the claim follows from Remark \ref{REMAk}. Otherwise, assume $\mathcal{H}^{n-1}(\Omega\cap\partial^*E)>0$. It suffices to show
\begin{align}\label{nikl}
\mathcal{M}_{C^{\prime}}(\partial E; \Omega)^*\leq\frac{1}{2}\big(\textup{Per}_{h_{C^{\prime}}}(E; \Omega)+\textup{Per}_{h_{C^{\prime}}}(\Omega\setminus E; \Omega)\big)\leq \mathcal{M}_{C^{\prime}}(\partial E; \Omega)_*.
\end{align}
The second inequality in (\ref{nikl}) always holds (see Proposition \ref{eq:4}), so it suffices to prove the first inequality. We decompose $\Omega\cap\partial^*E$ into finitely many Lipschitz pieces and a small remainder. The Lipschitz pieces are handled by Lemma \ref{crucial}, while the remainder is controlled by the Besicovitch covering theorem.

Since $\Omega\cap\partial^* E$ is countably $\mathcal{H}^{n-1}$-rectifiable and has finite $\mathcal{H}^{n-1}$-measure, there exists a sequence $\{S_k\}_{k=1}^{\infty}$ of pairwise disjoint compact subsets of $\Omega\cap\partial^* E$ that covers $\Omega\cap\partial^* E$ up to an $\mathcal{H}^{n-1}$-negligible set, each contained in the graph of a real-valued Lipschitz function (see Remark \ref{R1}). Since $\mathcal{H}^{n-1}(\Omega\cap\partial^* E)>0$, at least one of the sets $S_k$ is nonempty. Fix $k_0\in\mathbb{N}$ such that $S_{k_0}\neq\emptyset$. Then $\bigcup_{k=1}^K S_k\neq\emptyset$ for every $K\geq k_0$, so the distance function
\[
x\mapsto \dist\Bigl(x,\bigcup_{k=1}^K S_k\Bigr)
\]
is well defined.

We now prove the first inequality in (\ref{nikl}). Observe
\begin{equation*}
\begin{split}
\mathcal{M}_C(\partial E; \Omega)&=\frac{1}{2}\big(\textup{Per}_{h_{C}}(E; \Omega)+\textup{Per}_{h_{C}}(\Omega\setminus E; \Omega)\big)\\
&=\frac{1}{2}\int_{\partial^* E\cap\Omega}\big(h_C(\nu_E)+h_C(-\nu_E)\big)\,d\mathcal{H}^{n-1}\\
&=\sum_{k\in\mathbb{N}}\frac{1}{2}\int_{S_k}\big(h_C(\nu_E)+h_C(-\nu_E)\big)\,d\mathcal{H}^{n-1}=\sum_{k\in\mathbb{N}}\mathcal{M}_C(S_k; \Omega),
\end{split}
\end{equation*}
where the last equality follows from Lemma \ref{crucial}. For a given $\varepsilon>0$, there exists $K\geq k_0$ such that
\[
\mathcal{M}_C(\partial E; \Omega)-\sum_{k=1}^K\mathcal{M}_C(S_k; \Omega)<\varepsilon^{n}.
\]

Since $\{S_k\}_{k=1}^K$ is a finite family of pairwise disjoint compact subsets of $\Omega$, their mutual distances are positive. Therefore, for all sufficiently small $r>0$, the sets $S_k\oplus rC$, $k=1,\dots,K$, are pairwise disjoint. Hence
\begin{equation*}
\sum_{k=1}^K\mathcal{M}_C(S_k; \Omega)=\lim_{r\to0_+}\frac{1}{2r}\lambda^n\bigg(\Big(\bigcup_{k=1}^K S_k\oplus rC\Big)\cap\Omega\bigg),
\end{equation*}
so that
\begin{equation}\label{rum}
\lim_{r\to0_+}\frac{1}{2r}\lambda^n\Bigg(\Omega\cap\bigg(\Big((\Omega\cap\partial E)\oplus rC\Big)\setminus\Big(\bigcup_{k=1}^K S_k\oplus rC\Big)\bigg)\Bigg)<\varepsilon^n.
\end{equation}

Let positive constants $a$, $b$ and $c$ satisfy $B(0,a)\subseteq C$ and $B(0, c)\subseteq C^{\prime}\subseteq B(0,b)$, and define
\[
E_{r, \varepsilon}\coloneq \Big\{x\in \Omega\cap\partial E  : \dist\Big(x, \bigcup_{k=1}^K S_k\Big)>2\diam(C)r\varepsilon \Big\}\cap\Omega_{r, \varepsilon},
\]
where $\Omega_{r, \varepsilon}\coloneq \{x\in\Omega : \dist(x, \partial\Omega)>ar\varepsilon\}$ if $\partial\Omega\neq\emptyset$, and $\Omega_{r, \varepsilon}\coloneq \Omega$ otherwise. 

By Theorem \ref{bes}, there exists an at most countable set $I_{r, \varepsilon}\subseteq E_{r, \varepsilon}$ and $N\in\mathbb{N}$ such that
\begin{align}
    E_{r, \varepsilon}\subseteq \bigcup_{x\in I_{r, \varepsilon}} B(x, ar\varepsilon)&&\text{and}&&\sum_{x\in I_{r, \varepsilon}}\chi_{B(x, ar\varepsilon)}\leq N.
\end{align}
We have \begin{equation*}
    \begin{split}
        \mathcal{H}^0(I_{r, \varepsilon})\omega_n (ar\varepsilon)^n&=\sum_{x\in I_{r, \varepsilon}}\omega_n (ar\varepsilon)^n\\&=\sum_{x\in I_{r, \varepsilon}}\lambda^n\big(B(x, ar\varepsilon)\big)\\&\leq N\lambda^n\Big(\bigcup_{x\in I_{r, \varepsilon}}B(x, ar\varepsilon)\Big).
    \end{split}
\end{equation*} 
Notice
\[
\bigcup_{x\in I_{r, \varepsilon}} B(x, ar\varepsilon)\subseteq E_{r, \varepsilon}\oplus B(0, ar\varepsilon) \subseteq \bigg(\big((\Omega\cap\partial E)\oplus r\varepsilon C \big)\setminus \Big(\bigcup_{k=1}^K S_k\oplus r\varepsilon C\Big)\bigg)\cap\Omega.
\]
Since \begin{equation*}\begin{split}
\varepsilon^n&>\lim_{r\to0_+}\frac{1}{2r}\lambda^n\Bigg(\Omega\cap\bigg(\Big((\Omega\cap\partial E)\oplus rC\Big)\setminus\Big(\bigcup_{k=1}^K S_k\oplus rC\Big)\bigg)\Bigg)\\&=\lim_{r\to0_+}\frac{1}{2r\varepsilon}\lambda^n\Bigg(\Omega\cap\bigg(\Big((\Omega\cap\partial E)\oplus r\varepsilon C\Big)\setminus\Big(\bigcup_{k=1}^K S_k\oplus r\varepsilon C\Big)\bigg)\Bigg),
\end{split}
\end{equation*} inequality \eqref{rum} implies
\begin{equation}\label{frf}\limsup_{r\to0_+}\mathcal{H}^0(I_{r, \varepsilon}) r^{n-1}\leq \frac{2N\varepsilon}{a^n\omega_n}.
\end{equation}

We have $$\Omega\cap\partial E\subseteq \big((\Omega\cap\partial E)\setminus E_{r, \varepsilon}\big)\cup E_{r, \varepsilon}, \quad r>0.$$

On the one hand,
\[
(\partial E\cap\Omega)\setminus E_{r, \varepsilon}\subseteq\Big\{x\in \Omega\cap\partial E : \dist\Big(x, \bigcup_{k=1} ^K S_k\Big)\leq 2\diam(C)\varepsilon r\Big\}\cup\Omega^{\prime}_{r, \varepsilon}, \quad r>0,
\] where $$\Omega^{\prime}_{r, \varepsilon}\coloneq\{x\in \Omega\cap\partial E:\dist(x, \partial\Omega)\leq ar\varepsilon\}, \quad r>0,$$ provided that $\partial\Omega\neq\emptyset;$ otherwise, for $r>0$, put $\Omega^{\prime}_{r, \varepsilon}\coloneq\emptyset$.

Since $\Omega^{\prime}_{r, \varepsilon}\subseteq (\partial\Omega)\oplus B(0, ar\varepsilon)$ and $\bigcup_{k=1}^K S_k\subseteq\Omega$, we see that for any small enough $r>0$  $$(\Omega^{\prime}_{r, \varepsilon}\oplus rC^{\prime})\cap \Big(\bigcup_{k=1}^K S_k\oplus ba^{-1}rC\Big)=\emptyset,$$ whence $$(\Omega^{\prime}_{r, \varepsilon}\oplus rC^{\prime})\subseteq \big((\Omega\cap\partial E)\oplus ba^{-1}rC\big)\setminus\Big(\bigcup_{k=1}^K S_k\oplus ba^{-1}rC\Big),$$ which coupled with \eqref{rum} ensures that \begin{equation}\label{rumík}
    \limsup_{r\to0_+}\frac{1}{2r}\lambda^n\big((\Omega^{\prime}_{r, \varepsilon}\oplus rC^{\prime})\cap\Omega\big)<\varepsilon^n b^{-1}a
\end{equation} Furthermore, \begin{equation}\label{rumcajs}
    \begin{split}
        \lambda^n\bigg(\Big(\bigcup_{k=1}^K S_k\oplus B\big(0, 3\diam(C)\varepsilon r\big)\Big) \oplus rC^{\prime}\bigg)\leq \lambda^n\Big(\bigcup_{k=1}^K S_k\oplus r\big(1+c^{-1}3\diam(C)\varepsilon\big)C^{\prime}\Big).
    \end{split}
\end{equation}

On the other hand, 
\[
E_{r, \varepsilon} \oplus rC^{\prime}\subseteq \bigcup_{x\in I_{r, \varepsilon}} B(x, rb),
\]
whence, using \eqref{frf}, \begin{equation}\label{konec}
    \limsup_{r\to0_+}\frac{1}{2r}\lambda^n(E_{r, \varepsilon} \oplus rC^{\prime})\leq\frac{\omega_n b^n}{2}\limsup_{r\to0_+}\mathcal{H}^0(I_{r, \varepsilon})r^{n-1}\leq \frac{N b^n\varepsilon}{a^n}.\end{equation}

Finally, taking into account \eqref{rumík}, \eqref{rumcajs} and \eqref{konec}, we obtain
\[
\begin{split}
\mathcal{M}_{C^{\prime}}(\partial E; \Omega)^* &\leq \limsup_{r\to0_+}\frac{1}{2r}\lambda^n\bigg(\Big(\big((\Omega\cap \partial E)\setminus E_{r, \varepsilon}\big)\oplus rC^{\prime}\Big)\cap\Omega\bigg)\\&\quad\quad+\limsup_{r\to0_+}\frac{1}{2r}\lambda^n\big((E_{r, \varepsilon}\oplus rC^{\prime})\cap\Omega\big)\\
&\leq\limsup_{r\to0_+}\frac{1}{2r}\lambda^n\big((\Omega^{\prime}_{r, \varepsilon}\oplus rC^{\prime})\cap\Omega\big)\\&\quad\quad+\limsup_{r\to0_+}\frac{1}{2r}\lambda^n\Big(\bigcup_{k=1}^K S_k\oplus r\big(1+c^{-1}3\diam(C)\varepsilon\big)C^{\prime}\Big)\\&\quad\quad+\limsup_{r\to0_+}\frac{1}{2r}\lambda^n\big((E_{r, \varepsilon}\oplus rC^{\prime})\cap\Omega\big)\\&\leq \varepsilon^n b^{-1}a+\frac{Nb^n\varepsilon }{a^n} + \big(1+c^{-1}3\diam(C)\varepsilon\big)\sum_{k=1}^K\mathcal{M}_{C^{\prime}}(S_k; \Omega)\\&\leq \varepsilon^n b^{-1}a+\frac{Nb^n\varepsilon }{a^n} + \big(1+c^{-1}3\diam(C)\varepsilon\big)\sum_{k=1}^{\infty}\mathcal{M}_{C^{\prime}}(S_k; \Omega)\\
&= \varepsilon^n b^{-1}a+\frac{Nb^n\varepsilon }{a^n}\\&\quad\quad+ \big(1+c^{-1}3\diam(C)\varepsilon\big)\frac{1}{2}\big(\textup{Per}_{h_{C^{\prime}}}(E; \Omega)+\textup{Per}_{h_{C^{\prime}}}(\Omega\setminus E; \Omega)\big).
\end{split}
\]

Since $\varepsilon>0$ was arbitrary, we conclude
\[
\mathcal{M}_{C^{\prime}}(\partial E; \Omega)^* \leq \frac{1}{2}\big(\textup{Per}_{h_{C^{\prime}}}(E; \Omega)+\textup{Per}_{h_{C^{\prime}}}(\Omega\setminus E; \Omega)\big),
\]
which completes the proof.
\end{proof}

Using Theorem \ref{theoremExistence} and Proposition \ref{simonka}, we obtain:

\begin{corollary}\label{cor}
Let $\Omega\subseteq\mathbb{R}^n$ be open and $E\subseteq\Omega$ a set of finite perimeter with $\overline{E}\subseteq\Omega$, and let $C, C^{\prime}\in\mathcal{C}^n_0$. Then
\[
\mathcal{SM}_C(E; \Omega)=\textup{Per}_{h_C}(E; \Omega), \quad \mathcal{SM}_C(\Omega\setminus E; \Omega)=\textup{Per}_{h_C}(\Omega\setminus E; \Omega)
\]
if and only if
\[
\mathcal{SM}_{C^{\prime}}(E; \Omega)=\textup{Per}_{h_{C^{\prime}}}(E; \Omega), \quad \mathcal{SM}_{C^{\prime}}(\Omega\setminus E; \Omega)=\textup{Per}_{h_{C^{\prime}}}(\Omega\setminus E; \Omega).
\]
\end{corollary}

\medskip
Thus, under the natural assumption $\overline{E}\subseteq\Omega$, the validity of the expected outer Minkowski formula is independent of the chosen anisotropy.

In view of Corollary \ref{cor}, if $\mathcal{SM}_{C}(E; \Omega)=\textup{Per}_{h_{C}}(E; \Omega)$ and $\mathcal{SM}_{C}(\Omega\setminus E; \Omega)=\textup{Per}_{h_{C}}(\Omega\setminus E; \Omega)$, then $\mathcal{SM}(E; \Omega)=\Per(E; \Omega)$ and $\mathcal{SM}(\Omega\setminus E; \Omega)=\Per(E; \Omega)$. This yields a weaker converse statement than the one suggested in \cite[Remark 3.5]{chambolle}.

\begin{rem}
The existence of the $C$-anisotropic outer Minkowski content of a set of finite perimeter $E$ in $\Omega$ does \emph{not} guarantee the existence of the $C$-anisotropic outer Minkowski content of $\Omega\setminus E$. For a counterexample, take $E\coloneq\mathbb{R}\setminus \{0\}$ and $\Omega\coloneq\mathbb{R}$. Then, for any $C\in\mathcal{C}^1_0$,
\[
\mathcal{SM}_C(E; \mathbb{R})=0=\textup{Per}_{h_C}(E; \mathbb{R}),
\]
while
\[
\mathcal{SM}_C(\{0\}; \mathbb{R}) = \diam(C) \neq 0 = \textup{Per}_{h_C}(\{0\}; \mathbb{R}).
\]
A similar construction works for any open, nonempty $\Omega\subseteq \mathbb{R}^n$.
\end{rem}

\medskip

Now consider the functional $\mathfrak{M}_C$ and the reduced boundary of $E$. Suppose
\[
\mathfrak{M}_C(\partial^* E; \Omega)=\frac{1}{2}\big(\textup{Per}_{h_C}(E; \Omega)+\textup{Per}_{h_C}(\Omega\setminus E; \Omega)\big).
\]
Does this imply
\[
\mathcal{SM}_C(E; \Omega)=\textup{Per}_{h_C}(E; \Omega)?
\]
Unfortunately, this implication is false (see \cite[Remark 3.6]{chambolle}). For example, take $E\coloneq \mathbb{Q}$ and $\Omega\coloneq \mathbb{R}$. Likewise, $\mathcal{M}_C$ cannot replace $\mathfrak{M}_C$ in this implication. Nevertheless, the following improvement holds:

\begin{proposition}\label{reduced}
Let $\Omega\subseteq\mathbb{R}^n$ be open, $E\subseteq\Omega$ a set of finite perimeter in $\Omega$, and $C\in\mathcal{C}^n_0$. If
\[
\mathfrak{M}_C(\partial^* E; \Omega)=\frac{1}{2}\big(\textup{Per}_{h_C}(E; \Omega)+\textup{Per}_{h_C}(\Omega\setminus E; \Omega)\big),
\]
then
\[
\mathcal{SM}_C(E^1; \Omega)=\textup{Per}_{h_C}(E; \Omega), \quad \mathcal{SM}_C(E^0; \Omega)=\textup{Per}_{h_C}(\Omega\setminus E; \Omega).
\]
Conversely, if
\[
\mathcal{SM}_C(E^1; \Omega)=\textup{Per}_{h_C}(E; \Omega), \quad \mathcal{SM}_C(E^0; \Omega)=\textup{Per}_{h_C}(\Omega\setminus E; \Omega),
\]
then
\[
\mathcal{M}_C(\partial^* E; \Omega)=\frac{1}{2}\big(\textup{Per}_{h_C}(E; \Omega)+\textup{Per}_{h_C}(\Omega\setminus E; \Omega)\big).
\]
\end{proposition}

\begin{proof}
Remark \ref{remarko} gives $\overline{\partial^* E} = \partial(E^1) = \partial(E^0)$. If $\partial^* E = \emptyset$, the result is trivial. Otherwise, from (\ref{eq:54}),
\[
\mathfrak{M}_{\varepsilon, C}(\partial^* E; \Omega)=\mathfrak{M}_{\varepsilon, C}(\partial(E^1); \Omega)=\mathfrak{M}_{\varepsilon, C}(\partial(E^0); \Omega).
\]
 Thus both $\mathfrak{M}_C(\partial(E^1); \Omega)$ and $\mathfrak{M}_C(\partial(E^0); \Omega)$ exist. Applying Remark \ref{implik} to $E^1$ and to $E^0$, we conclude that both $\mathcal{SM}_C(E^1; \Omega)$ and $\mathcal{SM}_C(E^0; \Omega)$ exist.

Conversely, assume that
\[
\mathcal{SM}_C(E^1; \Omega)=\textup{Per}_{h_C}(E; \Omega), 
\qquad
\mathcal{SM}_C(E^0; \Omega)=\textup{Per}_{h_C}(\Omega\setminus E; \Omega).
\]
By (\ref{min4}), we also have
\[
\mathcal{SM}_C(\Omega\setminus E^1; \Omega)=\mathcal{SM}_C(E^0; \Omega),
\qquad
\mathcal{SM}_C(\Omega\setminus E^0; \Omega)=\mathcal{SM}_C(E^1; \Omega).
\]
Hence both $\mathscr{M}_C(E^1; \Omega)$ and $\mathscr{M}_C(E^0; \Omega)$ exist. By Proposition \ref{simonka}, both
\[
\mathcal{M}_C(\partial(E^1); \Omega)
\quad\text{and}\quad
\mathcal{M}_C(\partial(E^0); \Omega)
\]
exist. Since $\lambda^n((E^1\Delta E)\cap\Omega)=0$, we have
\[
\textup{Per}_{h_C}(E^1;\Omega)=\textup{Per}_{h_C}(E;\Omega),
\qquad
\textup{Per}_{h_C}(\Omega\setminus E^1;\Omega)=\textup{Per}_{h_C}(\Omega\setminus E;\Omega).
\]
Therefore, Remark \ref{EWQ} applied to $E^1$ yields
\[
\mathcal{M}_C(\partial^* (E^1); \Omega)
=
\frac{1}{2}\big(\textup{Per}_{h_C}(E; \Omega)+\textup{Per}_{h_C}(\Omega\setminus E; \Omega)\big).
\]
Since $\partial^*(E^1)=\partial^*E$, the claim follows.
\end{proof}

\begin{rem}
    Proposition \ref{reduced} together with (\ref{min4}) yields the following chain of implications:
    \begin{gather*} 
 \mathfrak{M}_{C}(\partial^* E; \Omega)\text{ exists}\implies\mathscr{M}_{C}(E^1; \Omega)\text{ exists}\implies\mathcal{M}_{C}(\partial^* E; \Omega)\text{ exists}\\
                    \Updownarrow                                         \\
    \mathcal{SM}_C(E^1; \Omega) \text{ and } \mathcal{SM}_C(E^0; \Omega)\text{ exist}
\end{gather*} 
\end{rem}

\end{document}